\documentclass[12pt]{amsart}

\usepackage[utf8]{inputenc}
\usepackage[T1]{fontenc}
\usepackage[a4paper]{geometry}
\usepackage{verbatim}
\usepackage{amsthm}
\usepackage{amstext}
\usepackage{amssymb}
\usepackage{graphicx}

\numberwithin{equation}{section}
\numberwithin{figure}{section}
\theoremstyle{remark}
\newtheorem*{acknowledgement*}{Acknowledgement}
\theoremstyle{plain}
\newtheorem{thm}{Theorem}
\theoremstyle{definition}
\newtheorem{defn}[thm]{Definition}
\theoremstyle{plain}

\theoremstyle{definition}
\newtheorem{example}[thm]{Example}
\theoremstyle{remark}

\theoremstyle{plain}

\theoremstyle{plain}
\newtheorem{prop}[thm]{Proposition}

\newcommand{\R}{\mathbb R}
\newcommand{\C}{\mathbb C}

\newcommand{\micro}{\bf mic}
\newcommand{\ext}{\bf ext}

\newcommand{\vect}{\bf Vect}
\newcommand{\Sympl}{\bf Sympl}
\newcommand{\poisson}{\bf Poisson}

\newcommand{\neutral}{\textbf{E}}
\newcommand{\neutralmorph}{\textbf{e}}
\newcommand{\Mon}[1]{{\bf Mon}(#1)}
\newcommand{\mon}{{M}}

\newcommand{\ExtSympl}{{\Sympl}^{\ext}}
\newcommand{\MicroSympl}{{\Sympl}_{\micro}}
\newcommand{\ExtMicroSympl}{{\Sympl}_{\micro}^{\ext}}
\newcommand{\SymplGpd}{\bf SymplGpd}
\newcommand{\ExtSymplGpd}{{\SymplGpd}^{\ext}}

\newcommand{\ExtMicroSymplGpd}{{\SymplGpd}_{\micro}^{\ext}}

\newcommand{\Cot}{{T}^*}
\newcommand{\id}{\operatorname{id}}
\newcommand{\graph}{\operatorname{gr}}

\begin{document}

\title{Symplectic microgeometry III: monoids}

\author[A. S. Cattaneo, B. Dherin and A. Weinstein]{Alberto S. Cattaneo, Benoit Dherin and Alan Weinstein}

\address{Institut f\"ur Mathematik\\
Universit\"at Z\"urich--Irchel\\
Winterthurerstrasse 190, CH-8057 Z\"urich\\
Switzerland}
\email{alberto.cattaneo@math.uzh.ch}

\address{ICMC-USP\\
Universidade de S\~ao Paulo\\
Avenida Trabalhador S\~ao-carlense, 400, Centro\\
CEP: 13566-590, São Carlos, SP\\
Brazil}
\email{dherin@icmc.usp.br}

\address{Department of Mathematics\\
University of California\\
Berkeley, CA 94720-3840}
\email{alanw@math.berkeley.edu}


\begin{abstract}
We show that the category of Poisson manifolds and Poisson maps, the category of symplectic microgroupoids and lagrangian submicrogroupoids (as morphisms), and the category of monoids and monoid morphisms in the microsymplectic category are equivalent symmetric monoidal categories.  
\end{abstract}

\maketitle

\tableofcontents{}

\section{Introduction}

This paper is a step toward a geometric and functorial approach to
Poisson manifold quantization. At its core is a micro-version of the
geometric approach to the quantization of Poisson manifolds by symplectic
groupoids, as developed in \cite{CDW1987,karasev1987,KM1993,weinstein1991,zakrzewski1990I,zakrzewski1990II}.
This micro-version is essentially obtained by replacing groupoids
by groupoid germs (or ``microgroupoids'') and canonical relations
by symplectic micromorphisms, which are the special canonical relation
germs introduced in \cite{CDW2010,CDW2011}. After these replacements,
geometric constructions that are only functorial-looking in the world
of groupoids and canonical relations become honest functors in the
world of microgroupoids and symplectic micromorphisms. We will use
quotation marks to remind us when we are not dealing with honest categories
and functors. 

In this paper, we are mostly concerned with developing a functorial
micro-version of two constructions. The first construction is a symplectization
``functor'' $\Sigma$, as in Fernandes \cite{fernandes2006}. It replaces
a Poisson manifold with its source 1-connected symplectic groupoid
and a Poisson map with the lagrangian subgroupoid integrating its
graph, which is regarded as a coisotropic submanifold. The second
functorial-looking construction, called $M$, replaces a symplectic
groupoid with a monoid object in the symplectic ``category'' (i.e.,
the category of symplectic manifolds and canonical relations) and
a lagrangian subgroupoid in the product of two symplectic groupoids
with a monoid map. 

In the micro-world, $\Sigma$ and $M$ become honest functors, both
of which are equivalences of symmetric monoidal categories. Moreover,
they have additional advantages over their macro-world counterparts:
\begin{itemize}
\item The domain of $\Sigma$ is now the category of \textsl{all} Poisson
manifolds and Poisson maps, as opposed to the macro-world where $\Sigma$
is only defined for the class of \textsl{integrable} Poisson manifolds;
\item a monoid object in the microsymplectic category characterizes the
symplectic microgroupoid it comes from completely, as opposed to the
macro-world where the inverse map is not part of the monoid structure
alone (see Example \ref{exam: sing. gpd} and \cite{Canez2011} for
similar issues related to the inverse map).
\end{itemize}
The composition $Z=M\circ\Sigma$ (which we call the ``Zakrzewski
functor'') establishes an equivalence of monoidal categories between
the category of Poisson manifolds and Poisson maps and the category
of monoids and monoid morphisms in the microsymplectic category. This
is the main result of this paper. In itself, it gives a categorical
formulation of Poisson geometry; it allows us to replace, in a functorial
way, the Poisson bivector field with a commutative diagram in a monoidal
category. Moreover, it gives an appropriate categorical framework
for Poisson manifold functorial quantization by symplectic (micro)groupoid
methods. This framework may also be relevant for functorial aspects of star-product dequantization as in \cite{Karabegov2003, Karabegov2004, Karabegov2005, Karabegov2006}. 

\begin{acknowledgement*}
We thank Domenico Fiorenza for useful suggestions. A.S.C. acknowledges partial support from SNF Grant 20-131813. B.D.
acknowledges partial support from SNF Grant PA002-113136, NWO Grant
613.000.602 and FAPESP grant 2010/15069-8, as well as the hospitality
of the ICMC of the University of S\~ao Paulo, in S\~ao Carlos. B.D. thanks
Ittay Weiss for inspiring discussions about monoids in Rel. In addition,
B.D. thanks Ieke Moerdijk and Marius Crainic for providing a highly
stimulating research environment at Utrecht University, where part
of this paper was done. A.W. acknowledges partial support from NSF
grant DMS-0707137 and the hospitality of the Institut Math\'ematique
de Jussieu.
\end{acknowledgement*}

\section{Preliminaries}

\subsection{Poisson geometry}

We fix here some notations and terminology. 

A \textbf{Poisson manifold} is a smooth manifold $A$ whose algebra
of smooth functions is endowed with a Poisson bracket; i.e. a Lie
bracket $\{\,,\,\}$ on the space of smooth functions $C^{\infty}(A)$
that is a derivation in each argument. Thus, this bracket is completely
determined by a bivector field $\Pi\in\Gamma(\wedge^{2}TA)$ defined
by $\{f,g\}(x)=\langle df\wedge dg,\Pi\rangle,$ where $f,g\in C^{\infty}(A)$.
A \textbf{Poisson map} $\phi:B\rightarrow A$ is a smooth map such
that $\phi^{*}:C^{\infty}(A)\rightarrow C^{\infty}(B)$ respects the
Poisson brackets. 

We denote by $\poisson$ the category of Poisson manifolds and Poisson
maps. This category is monoidal (see \cite{MacLane1998} for a general
reference on monoidal categories). Namely, the tensor product of two
Poisson manifolds, denoted by $A\otimes B$, is defined by endowing
the cartesian product $A\times B$ of the manifold with the Poisson
bivector field $\Pi_{A}\oplus\Pi_{B}$ given by the sum of the Poisson
bivector fields. On Poisson maps, the tensor product is simply the
usual cartesian product of maps. The unit object of the category is
the one-point manifold endowed with its unique (zero) Poisson structure.
$\poisson$ is in fact a symmetric monoidal category, whose symmetries
$A\times B\rightarrow B\times A$ are the usual factor permutations
$\epsilon_{A,B}(a,b)=(b,a)$. 

A \textbf{coisotropic submanifold} $C$ of a Poisson manifold $(A,\Pi)$
is a submanifold satisfying the condition $\Pi^{\sharp}(N^{*}(C))\subset TC$,
where $\Pi^{\sharp}:\Cot A\rightarrow TA$ is the vector bundle map
$\Pi^{\sharp}(\nu)=\Pi(\nu,\,-\,)$ and $N^{*}(C)\subset\Cot A$ is
the conormal bundle of $C$. Coisotropic submanifolds play a role
similar to that of lagrangian submanifolds in symplectic geometry:
Namely, a smooth map $\phi:B\rightarrow A$ is Poisson if and only
if its graph, denoted by $\graph\phi$, is a coisotropic submanifold
of $A\otimes\overline{B}$, where $\overline{B}$ denotes the \textbf{dual
Poisson manifold} of $(B,\Pi)$, which is the manifold $B$ endowed
with the opposite Poisson structure $-\Pi$ .

There is a notion of integration for Poisson manifolds by symplectic
groupoids that generalizes the integration of Lie algebras by Lie
groups. A \textbf{symplectic groupoid} is a Lie groupoid $G\rightrightarrows A$
(see \cite{mackenzie1987} for a general reference on Lie groupoids)
whose total space $G$ is a symplectic manifold and such that the
graph of the groupoid product $m:G\times_{A}G\longrightarrow G$ is
a lagrangian submanifold of $\overline{G}\times\overline{G}\times G$.
A symplectic groupoid endows its unit space $A$ with a unique Poisson
structure such that the target map $s:G\rightarrow A$ (resp. the
source map $t:G\rightarrow A$) is Poisson (resp. anti-Poisson). We
say that $G\rightrightarrows A$ \textbf{integrates} the Poisson manifold
$(A,\Pi)$. A morphism of symplectic groupoids is a morphism of groupoids
that is a symplectomorphism. 

Let us write $\SymplGpd$ for the category of symplectic groupoids
and symplectic groupoid morphisms (i.e. symplectomorphisms that also
are groupoid morphisms). As in the category of Poisson manifolds,
the cartesian product of manifolds endows the category of symplectic
groupoids with a symmetric monoidal structure. One also defines the
\textbf{dual symplectic groupoid} $\overline{G\rightrightarrows A}$
of a symplectic groupoid $(G,\omega)\rightrightarrows A$ as the the
symplectic groupoid $(G,-\omega)\rightrightarrows A$ (which can also
be obtained by interchanging the source and target while {}``reversing''
the product law). If the symplectic groupoid integrates $A$, its
dual integrates $\overline{A}$. Similarly, if the symplectic groupoid
$G\rightrightarrows A$ integrates the Poisson manifold $(A,\Pi_{A})$
and the symplectic groupoid $H\rightrightarrows B$ integrates the
Poisson manifold $(B,\Pi_{B})$, then the symplectic groupoid product
$G\times H\rightrightarrows A\times B$ integrates the Poisson manifold
product $(A\times B,\Pi_{A}\oplus\Pi_{B})$. 

Not all Poisson manifolds can be integrated by symplectic groupoids.
The ones which can are called \textbf{integrable}. All the source
$1$-connected symplectic groupoids integrating an integrable Poisson
manifold are isomorphic to the one resulting from the following construction
in terms of homotopy classes of paths \cite{CF2001,CF2004}. Let $(A,\Pi)$
be an integrable Poisson manifold. Consider the space $\mathcal{P}(\Cot A)$
of cotangent paths; that is, the set of paths $g:[0,1]\rightarrow\Cot A$
such that
\[
p(t)=\Pi^{\sharp}(x(t))\dot{x}(t),\,\textrm{where }g(t)=(p(t),x(t)),
\]
where $x:[0,1]\rightarrow A$ is the base map. The source $1$-connected
symplectic groupoid of $(A,\Pi)$ can be realized as the quotient
of the set of cotangent paths by a homotopy relation $\thicksim$
that fixes the end-points of the $x$-component of the cotangent path
(see \cite{CF2001,CF2004} for a definition of this relation). We
denote by $\Sigma(A)$ this quotient and by $[g]$ the homotopy class
of $g$. $\Sigma(A)$ is a symplectic groupoid over $A$ whose source
and target maps $s,t:\Sigma(A)\rightrightarrows A$ are given by the
endpoints of the path projection on the base: $s([g])=x(0)$ and $t([g])=x(1)$.
The groupoid product is given by concatenation of paths $[g][g']=[gg']$,
where $g\in[g]$ and $g'\in[g']$ are two representatives whose ends
$t([g])=s([g'])$ match smoothly and where 
\[
(gg')(t)=\left\{ \begin{array}{cc}
2g(2t), & 0\leq t\leq\frac{1}{2},\\
2g'(2t-1), & \frac{1}{2}<t\leq1.
\end{array}\right.
\]
From now on, we will reserve the notation $\Sigma(A)\rightrightarrows A$
for the source $1$-connected symplectic groupoid integrating $(A,\Pi)$
coming from the construction above. 

Note that $\Sigma(A)$ always exists but not necessarily as a manifold:
for non-integrable Poisson manifolds, $\Sigma(A)$ can be realized
as a stack (see \cite{TZ2005})

\subsection{Symplectic categories}

Let us fix here a coherent notation for the various categories we
will be dealing with. First of all, we have the following:
\begin{itemize}
\item $\Sympl$ is the usual symplectic category of symplectic manifolds
and symplectomorphisms,
\item $\ExtSympl$ is the extended symplectic ``category'', where symplectomorphisms
are replaced by canonical relations.
\end{itemize}
While the former is an honest category, the latter is not: the composition
of two canonical relations may fail to produce again a canonical relation.
There are two honest categories obtained by considering symplectic
\textit{microfolds} instead of symplectic \textit{manifolds} (see
\cite{CDW2010,CDW2011} for more details)\textit{. }

A \textbf{symplectic microfold} $[M,A]$ is a germ of a symplectic
manifold $M$ around a lagrangian submanifold $A\subset M$ called
the \textbf{core} of the microfold. 

A \textbf{symplectic micromorphism} $([V],\phi):[M,A]\rightarrow[N,B]$
between two symplectic microfolds is a germ $[V,\graph\phi]$ of a
canonical relation $V\subset\overline{M}\times N$ around the graph
of a smooth map $\phi:B\rightarrow A$, that intersects the core $A\times B$
\textit{cleanly} in $\graph\phi$ (this definition is not the original
one given in \cite{CDW2010} but an equivalent one as stated in \cite[Thm. 17]{CDW2011}).
We denote by
\begin{itemize}
\item $\MicroSympl$ the category of symplectic microfolds and germs of
symplectomorphisms,
\item $\ExtMicroSympl$ the category of the symplectic microfolds and symplectic
micromorphisms. 
\end{itemize}
We call this last category the \textbf{microsymplectic category. }

Each of the categories above is symmetric monoidal. The symmetric
monoidal structure for $\Sympl$ and $\MicroSympl$ comes from the
usual cartesian product of sets and maps, while the symmetries $\epsilon_{M,N}$
are given by the usual factor permutations $M\times N\rightarrow N\times M$. 

As for $\ExtSympl$, the symmetric monoidal structure comes from the
one on the category of sets and \textit{binary relations}: The tensor
product on objects is still given by the cartesian product of the
underlying sets, while the tensor product between a canonical relation
$V$ from $M$ to $N$ and a canonical relation from $P$ to $Q$
is the canonical relation
\[
V\otimes W:=(\id_{M}\times\epsilon_{N,P}\times\id_{Q})(V\times W)
\]
from $M\times P$ to $N\times Q$; i.e., the canonical relation $V\times W$
with the middle factors permuted. The symmetries are given by the
graphs of the factor permutations
\[
\graph\epsilon_{M,N}:M\times N\rightarrow N\times M.
\]
The unit object $E$ is the one-point symplectic manifold, which we
will also denote by $\{\star\}$. 

We obtain the symmetric monoidal structure in the microsymplectic
category by going to the representatives. The tensor product of two
symplectic microfolds is given by the usual cartesian product of their
underlying representatives:
\[
[M,A]\otimes[N,B]:=[M\times N,A\times B].
\]
The tensor product of two symplectic micromorphisms is simply 
\[
([V],\phi)\otimes([W],\psi):=([V\otimes W],\phi\times\psi),
\]
where $V\otimes W$ is the tensor product of the canonical relation
representatives as above. Again, the symmetries are given by the graphs
of the factor permutation maps. For the unit object $E$, we can take
the symplectic microfold associated with the cotangent bundle of the
one-point manifold $\Cot\{\star\}\simeq\{0\}\times\{\star\}$. Observe
that, in the microsymplectic category, there is only one morphism
from the unit object to any given symplectic microfold:
\[
e_{M}:=([\{\star\}\times A],\, pr_{A}):\; E\rightarrow[M,A],
\]
where $pr_{A}$ is the unique map from $A$ to the one-point manifold.

\subsection{Categories of monoids\label{sub:Categories-of-monoids}}

Each of the symplectic categories considered in the previous paragraph
is symmetric monoidal, and thus possesses an associated category of
monoids. 

Recall that a \textbf{monoid object} (or monoid for short) in a monoidal
category $(\mathcal{C},\otimes,E)$ with neutral object $\neutral$
is a triple $(C,\mu,\neutralmorph)$ consisting of an object $C$,
a morphism $\mu:C\otimes C\rightarrow C$, called the product, and
a morphism $E:\neutral\rightarrow C$, called the unit. These morphisms
satisfy the associativity and unitality axioms: i.e., respectively,
\begin{gather}
	\mu\circ(\mu\otimes\id)=\mu\circ(\id\otimes\mu)\label{ass:ass-1}\\
	\mu\circ(e\otimes\id)=\mu\circ(\id\otimes\, e)=\id.\label{ass:neut-1}
\end{gather}
A \textbf{monoid morphism} from a monoid object $(C,\mu_{c},e_{c})$
to a monoid object $(D,\mu_{D},e_{D})$ is a morphism $T:C\rightarrow D$
that respects the products and the units: i.e., 
\begin{gather}
	T\circ\mu_{C}=\mu_{D}\circ(T\otimes T),\label{morph:ass-1}\\
	T\circ e_{C}=e_{D}\circ T.\label{morph:neut-1}
\end{gather}
The monoid objects in $\mathcal{C}$ together with the monoid morphisms
form a category $\Mon{\mathcal{C}}$. In the sequel, we will be particularly
interested in $\Mon{\ExtSympl}$, which is only a ``category'', and
its micro-version $\Mon{\ExtMicroSympl}$, which is an honest category.

The category of monoids in a symmetric monoidal category $(\mathcal{C},\otimes,\epsilon,E)$
is again monoidal symmetric. The tensor product on objects is given
by
\[
(C,\mu_{C},e_{C})\otimes(D,\mu_{D},e_{D}):=(C\otimes D,\mu_{C\otimes D},e_{C}\otimes e_{D}),
\]
where
\[
\mu_{C\otimes D}:=(\mu_{C}\otimes\mu_{D})\circ(\id_{C}\otimes\epsilon_{C,D}\otimes\id_{D}),
\]
while the tensor product on monoid morphisms is the tensor product
of the underlying morphisms in $\mathcal{C}$. Since a monoidal category
comes with an isomorphism $E\otimes E\rightarrow E$, the unit object
in $\mathcal{C}$ is also the unit object in its category of monoids.

\subsection{Functorial quantization}

Symplectic groupoids were first introduced as a potential tool to
be used along with geometric quantization in order to construct star
products for Poisson manifolds (\cite{CDW1987,karasev1987,KM1993,Landsman1993,weinstein1991,zakrzewski1990I,zakrzewski1990II}).
Namely, geometric quantization establishes a dictionary between symplectic
geometry and linear algebra%
\footnote{Actually, $Q(L)$ is a set of vectors in $Q(M)$ unless $L$ is ``enhanced''
by a half density. For simplicity, we leave aside this aspect of the
discussion, and we refer the reader to \cite{weinstein1996} for a
full account. %
} 
\[
\begin{array}{ccc}
\textrm{symplectic manifold}\quad M & \rightsquigarrow & \textrm{vector space}\quad Q(M)\\
\textrm{lagrangian submanifold}\quad L\subset M & \rightsquigarrow & \textrm{vector}\quad Q(L)\in Q(M)\\
\textrm{opposite}\quad\overline{M} & \rightsquigarrow & \textrm{dual}\quad Q(M)^{*}\\
\textrm{cartesian product}\quad M\times N & \rightsquigarrow & \textrm{tensor product}\quad Q(M)\otimes Q(N)
\end{array}
\]
implying the additional entry
\begin{eqnarray*}
\textrm{canonical relation}\quad L\subset\overline{M}\times N & \rightsquigarrow & \textrm{linear map}\quad Q(L):Q(M)\rightarrow Q(N).
\end{eqnarray*}
This suggests the possibility of a monoidal ``functor''
\begin{eqnarray*}
Q:\ExtSympl & \rightarrow & \vect_{\C},
\end{eqnarray*}
(the quotation marks remind us that the source of this ``functor''
is only a ``category''), which would induce a ``functor'' between
the corresponding categories of monoids 
\begin{eqnarray*}
Q:\Mon{\ExtSympl} & \overset{Q}{\longrightarrow} & {\bf Alg}_{\C},
\end{eqnarray*}
where ${\bf Alg}_{\C}$ is the category of monoids in $\vect_{\C}$,
or, in other words, the category of unital associative algebras and
unital algebra maps over $\C$. 

Now, the product graph $\graph m\subset\overline{G}\times\overline{G}\times G$
of a symplectic groupoid $G\rightrightarrows A$ together with its
unit space $A$ (regarded as a canonical relation from the one-point
symplectic manifold to $G$) produces a monoid object $(G,\graph m,A)$
in the symplectic ``category''. The image of this monoid object by
$Q$ should yield a unital associative algebra $\big(Q(G),Q(\graph m),Q(A)\big)$,
which is to be thought of as the quantum algebra that quantizes the
Poisson manifold $(A,\Pi)$ induced by the symplectic groupoid. 

The strategy above has been carried out successfully in some concrete
examples and special classes of integrable Poisson manifolds (\cite{Bieliavsky2002,BDS2009,Hawkins2008,GBJV1995,RT2008,weinstein1991,weinstein1994,weinstein1997}).
However, with this approach (as opposed to deformation quantization
for instance), non integrable Poisson manifolds are discarded from
the start, and functoriality issues, due to the ill-behaved composition
of canonical relations, have to be faced. As we shall see in the coming
sections, the use of the microsymplectic category improves the situation
in these respects. Namely, any monoidal functor 
\begin{eqnarray}
\ExtMicroSympl & \overset{Q}{\longrightarrow} & \vect_{\C}\label{Functor: quantization functor}
\end{eqnarray}
respecting duals and tensor products, would yield a general quantization
scheme%
\footnote{In reality, this strategy is too naive. One should first build an
intermediate category that has the same objects as the microsymplectic
category but whose morphisms are enhanced by half densities on the
lagrangian germs of the symplectic micromorphisms (see \cite{weinstein1996}). %
} for Poisson manifolds. The question of the functoriality of these
quantization schemes is closely related to the possibility of enlarging
the category of local symplectic groupoids into a category allowing
us to regard the integration of Poisson manifolds by local symplectic
groupoids as a functor. In fact, these ideas point to a local version
of the early work of Zakrzewski \cite{zakrzewski1990I,zakrzewski1990II}
and the more recent symplectization ``functor'' studied by Fernandes
in \cite{fernandes2006} in the global case. 

Let us conclude this paragraph by mentioning yet another way to circumvent
the functorial issues raised by the ill-defined composition of canonical
relations. Instead of the microsymplectic category, one could use
the WW-category introduced by Wehrheim and Woodward in \cite{WW2010},
which contains and enlarges the symplectic ``category'' so as to obtain
an honest category. A symplectic groupoid can be then identified with
a monoid object (endowed with some additional structures such
as a $*$-operation as in Zakrzewski \cite{zakrzewski1990I,zakrzewski1990II})
in this category.

\section{From Poisson manifolds to symplectic microgroupoids}

\subsection{Construction of $\Sigma$\label{sub: path construction}}

The extended symplectic ``category'' can be used as a target category
to ``symplectify'' Poisson geometry via the symplectization ``functor''
as studied by Fernandes in \cite{fernandes2006}. The general idea
is to replace an integrable Poisson manifold $(A,\Pi)$ by its source
1-connected symplectic groupoid $\Sigma(A)\rightrightarrows A$ as
described in the previous section. A Poisson map $\phi:(B,\Pi_{B})\rightarrow(A,\Pi_{A})$
is then replaced by the lagrangian subgroupoid $\Sigma(\phi)\subset\overline{\Sigma(A)}\times\Sigma(B)$
that integrates the graph of $\phi$ regarded as a coisotropic submanifold
of $\overline{\Pi_{B}}\times\Pi_{A}$ . This yields a ``functor''
\[
\Sigma:\poisson^{{\bf {int}}}\rightarrow\ExtSympl
\]
from the category of \textit{integrable} Poisson manifolds (i.e. the
Poisson manifolds that have a symplectic groupoid integrating them)
to the extended symplectic ``category'': the symplectization functor.

The definition of $\Sigma$ on morphisms rests on the following facts
(see \cite{cattaneo2004}):
\begin{itemize}
\item Let $G\rightrightarrows A$ be a source $1$-connected symplectic
groupoid over the Poisson manifold $A$ and let $L\rightrightarrows C$
be a (immersed) subgroupoid of $G\rightrightarrows A$, then $C$
is a coisotropic submanifold of $A$; conversely, there is a unique
(immersed) source 1-connected lagrangian subgroupoid that has a given
coisotropic submanifold as its unit space;
\item A smooth map $\phi:B\rightarrow A$ between the Poisson manifolds
$A$ and $B$ is Poisson iff its graph is a coisotropic submanifold
of $A\times\overline{B}$.
\end{itemize}
Therefore, we define the value of $\Sigma$ on a Poisson map $\phi:B\rightarrow A$
to be the unique lagrangian subgroupoid $\Sigma(\phi)\rightrightarrows\graph\phi$
of the symplectic groupoid product $\overline{\Sigma(A)}\times\Sigma(B)\rightrightarrows A\times B$
that integrates the coisotropic submanifold $\graph\phi$ of the Poisson
manifold $\overline{A}\times B$. 

Actually, there is a better suited ``category'' than $\ExtSympl$
as target category for the symplectization ``functor''. Namely, consider
the ``category''
\[
\ExtSymplGpd
\]
 whose objects are the source $1$-connected symplectic groupoids.
A morphism from $G\rightrightarrows A$ to $H\rightrightarrows B$
is a lagrangian subgroupoid $L\rightrightarrows C$ of the symplectic
groupoid product $\overline{G}\times H\rightrightarrows A\times B$.
The composition is given by the composition of the canonical relations
$L$ underlying the morphisms $L\rightrightarrows C$.

\subsection{Micro-version}

The symplectization functor as defined above suffers from two major
problems: it is only a ``functor'' and it is relevant only for the
class of integrable Poison manifolds. In order to define an honest
symplectization functor whose domain is the category of \textit{all}
Poisson manifolds, we need to consider the integration of Poisson
manifolds by \textit{local} symplectic groupoids instead of global
ones. What changes now is that \textit{every} Poisson manifold can
be integrated by a local symplectic groupoid (\cite{CDW1987,karasev1987}),
or more conveniently, by a \textit{symplectic microgroupoid}. 

Recall that a local groupoid $G\rightrightarrows A$ is, roughly,
a groupoid whose structure maps are defined only in a neighborhood
of the unit space (see \cite{CDW1987,mackenzie1987} for more details).
Two local groupoids over $A$ coinciding on a (smaller) neighborhood
of $A$ are equivalent for most purposes. Thus, the relevant information
is really contained in the ``groupoid germ'' induced by the local
groupoids. We are led to the following definition:

\begin{defn}
A (symplectic) microgroupoid $[G,A]\rightrightarrows A$ over $A$
is an equivalence class of local (symplectic) groupoids over $A$.
A (lagrangian) submicrogroupoid $[L,C]\rightrightarrows C$ of a (symplectic)
microgroupoid $[G,A]\rightrightarrows A$ is a (lagrangian) submicrofold
$[L,C]$ in $[G,A]$ such that it is also a submicrogroupoids over
$C$. 
\end{defn}

Now, every Poisson manifold $(A,\Pi)$ has a symplectic microgroupoid
integrating it (e.g. the one constructed by Karasev in \cite{karasev1987}),
which we will denote by
\[
[\Sigma(A),A]\rightrightarrows A,
\]
and which gives the value of the micro version of the symplectization
functor
\[
\Sigma:\poisson\rightarrow\ExtMicroSympl
\]
on objects. As for the morphism component of the functor, the local
lagrangian groupoid $\Sigma(\phi)\rightrightarrows\graph\phi$ integrating
the graph of a Poisson map $\phi:B\rightarrow A$ (as a coisotropic
manifold) yields a germ of a canonical relation $[\Sigma(\phi),\graph\phi]$
from the symplectic microfold $[\Sigma(A),A]$ to the symplectic microfold
$[\Sigma(B),B]$. Note that $[\Sigma(\phi),\graph\phi]$ also defines
a lagrangian submicrogroupoid over $\graph\phi$, which is unique.
To show that $\Sigma$ is well-defined, we still need to check that
$[\Sigma(\phi),\graph\phi]$ is a symplectic micromorphism and that
$\Sigma$ is functorial. 
\begin{prop}
\label{pro: lag.micgpds = sympl micmorph.}Let $[G,A]\rightrightarrows A$
and $[H,B]\rightrightarrows B$ be two symplectic microgroupoids.
Any lagrangian submicrogroupoid $[L,\graph\phi]\rightrightarrows\graph\phi$
of the symplectic microgroupoid product, and where $\phi:B\rightarrow A$
is a smooth map, yields the symplectic micromorphism $([L],\phi):[G,A]\rightarrow[H,B].$ 

Moreover, if we have another symplectic microgroupoid $[I,C]\rightrightarrows C$
and a lagrangian submicrogroupoid $([K],\psi):[H,B]\rightarrow[I,C]$,
the composition of the corresponding symplectic micromorphisms yields
a lagrangian submicrogroupoid
\[
[K\circ L,\graph\phi\circ\psi]\rightrightarrows\graph\phi\circ\psi.
\]
\end{prop}
\begin{proof}
Let $L\rightrightarrows\graph\phi$ be a local groupoid representative
of $[L,\graph\phi]\rightrightarrows\graph\phi$. Since $\graph\phi$
is the space of units of $L$, we have by definition that $L\cap(A\times B)=\graph\phi$.
We need to show that this intersection is clean (\cite[Thm. 17]{CDW2011}).
For this, consider the local tangent groupoid $TL\rightrightarrows T\graph\phi$.
Since $T\graph\phi$ is now the unit space of $TL$, we obtain for
the same reason as before that $TL\cap(TA\cap TB)=T\graph\phi$. 

The fact that $[K\circ L,\graph\phi\circ\psi]$ is a lagrangian microsubgroupoid
over $\graph\phi\circ\psi$ comes from the facts that $[L\times K,\graph\phi\times\graph\psi]$
is a lagrangian submicrogroupoid of 
\[
[G\times H\times H\times I,A\times B\times B\times C]\rightrightarrows A\times B\times B\times C
\]
and that all the defining properties of a lagrangian submicrogroupoid
are preserved by the reduction with respect to $G\times\Delta_{H}\times I$
(which is directly checked by taking representatives). 
\end{proof}
The functoriality of $\Sigma$ now follows from the uniqueness of
the lagrangian submicrogroupoid in
\[
[\overline{\Sigma(A)}\times\Sigma(B),A\times B]\rightrightarrows A\times B
\]
that integrates a given Poisson map $\phi:B\rightarrow A$, and the
fact that, from Proposition \ref{pro: lag.micgpds = sympl micmorph.},
both lagrangian submicrogroupoids 
\[
[\Sigma(\phi_{1}\circ\phi_{2}),\,\graph\phi_{1}\circ\phi_{2}]\rightrightarrows\graph\phi_{1}\circ\phi_{2}\leftleftarrows[\Sigma(\phi_{2})\circ\Sigma(\phi_{1}),\graph\phi_{1}\circ\phi_{2}]
\]
integrate the Poisson map $\phi_{1}\circ\phi_{2}$.

\subsection{Changing the target category}

Thanks to Proposition \ref{pro: lag.micgpds = sympl micmorph.},
we can consider the category
\[
\ExtMicroSymplGpd
\]
that has the symplectic microgroupoids as its objects and whose morphisms
from $[G,A]\rightrightarrows A$ to $[H,B]\rightrightarrows B$ are
the symplectic micromorphisms $([L],\phi)$ from $[G,A]$ to $[H,B]$
whose underlying lagrangian submicrofold are lagrangian submicrogroupoids
$[L,\graph\phi]\rightrightarrows\graph\phi$ of the symplectic microgroupoid
product. 

As in the macro case, this category can again be taken as the target
category of the symplectization functor. Namely, by Proposition \ref{pro: lag.micgpds = sympl micmorph.},
\[
\big(\big[\Sigma(\phi_{1})\circ\Sigma(\phi_{2})\big],\phi_{1}\circ\phi_{2}\big)
\]
is also a lagrangian subgroupoid that integrates the composition
of the Poisson maps $\phi_{1}$ and $\phi_{2}$. Now, in the micro
case, a nice thing happens:

\begin{thm}
The functor $\Sigma:\poisson\rightarrow\ExtMicroSymplGpd$ is an equivalence
of symmetric monoidal categories. 
\end{thm}

\begin{proof}
Since there is a one-to-one correspondence between the lagrangian
submicrogroupoids of the form
\begin{eqnarray*}
[L,\graph\phi]\rightrightarrows\graph\phi & \subset & [\overline{\Sigma(A)}\times\Sigma(B),A\times B]\rightrightarrows A\times B
\end{eqnarray*}
and the Poisson maps $\phi:B\rightarrow A$, we have that $\Sigma$
is full and faithful. The essential surjectivity of $\Sigma$ on the
objects follows from the fact that two symplectic microgroupoids that
integrate the same Poisson manifold are related by a germ of symplectic
microgroupoid isomorphism. In particular, for any symplectic microgroupoid
$[G,A]\rightrightarrows A$, there is a symplectic microgroupoid isomorphism
germ $[\Psi]:[G,A]\rightarrow[\Sigma(A),A]$, whose restriction to
the core is the identity. Now, one checks that $[\graph\Psi,\id_{A}]$
is a lagrangian submicrogroupoid over $\graph\id_{A}$. The monoidality
and symmetry properties of the functor come from the natural isomorphisms
$\Sigma(A\times B)\simeq\Sigma(A)\times\Sigma(B)$ and $\Sigma(\phi_{1}\times\phi_{2})\simeq\Sigma(\phi_{1})\times\Sigma(\phi_{2})$.
\end{proof}

\section{From symplectic microgroupoids to monoids}

Given a symplectic groupoid $G\rightrightarrows A$, the graph of
its product together with its space of units yield canonical relations
$\graph m:G\otimes G\rightarrow G$ and $e_{A}:\{\star\}\rightarrow G$,
where $e_{A}=\{\star\}\times A$. As noted in \cite{weinstein1996,CDW1987,zakrzewski1990II},
the triple
\[
M(G):=(G,\graph m,e_{A})
\]
 can be interpreted as a monoid object in the extended symplectic
``category''. In \cite{zakrzewski1990II} Zakrzewski constructed a
``functor''
\[
Z:\poisson^{{\bf int}}\rightarrow\Mon{\ExtSympl}
\]
directly from the category of integrable Poisson manifolds and complete
Poisson maps to the ``category'' of monoids in $\ExtSympl$ (although
using a quite different language). At the level of objects, this ``functor''
takes an integrable Poisson manifold $(A,\Pi)$ to the monoid object
$M(\Sigma(A))$ associated with its source $1$-connected symplectic
groupoid $\Sigma(A)\rightrightarrows A$. At the level of morphisms,
Zakrzewski obtained the canonical relation corresponding to the monoid
morphism directly out of the complete Poisson map by a flow construction. 

In this section, we give a micro-version of $Z$, which is an honest
functor and whose domain is the category of \textit{all} Poisson manifolds
and \textit{all} Poisson maps. We obtain $Z$ as factored through
the symplectization functor
\[
\poisson\overset{\Sigma}{\longrightarrow}\ExtMicroSymplGpd\overset{\mon}{\longrightarrow}\Mon{\ExtMicroSympl}.
\]
We show that it yields an equivalence of symmetric monoidal categories
between the category of Poisson manifolds and Poisson maps and the
category of monoids and monoid morphisms in the microsymplectic category
(Theorem \ref{thm: main}).

\subsection{Construction of $\mon$}

The association $\mon$ that takes a symplectic groupoid $G\rightrightarrows A$
to the corresponding monoid object $\mon(G)$ in the extended symplectic
``category'' has its corresponding version in the micro-world.

\begin{prop}
The graph of the product in a symplectic microgroupoid $[G,A]\rightrightarrows A$
is a symplectic micromorphism
\[
([\graph m],\Delta_{A}):[G,A]\otimes[G,A]\rightarrow[G,A],
\]
where $\Delta_{A}:A\rightarrow A\times A$ is the diagonal map. We
will simply denote it by $\graph[m]$. Moreover,
\[
\mon([G,A]):=([G,A],\graph[m],e_{A}),
\]
where $e_{A}=(\{\star\}\times A,pr_{A})$ is the unique symplectic
micromorphism from the unit object $E=\Cot\{\star\}$ to $[G,A]$,
is monoid object in the microsymplectic category.
\end{prop}

\begin{proof}
We need to check that $\graph m$ intersects $A^{3}$ cleanly in $\graph\Delta_{A}$.
The fact that $(\graph m)\cap A^{3}=\graph\Delta_{A}$ comes from
the facts that $m(a,a')=a''$ for $a,a',a''\in A$ iff $a=a'=a''$.
The cleanliness of the intersection comes from the repetition of this
argument for the tangent local groupoid $TG\rightrightarrows TA$.
The monoid object axioms follow directly from the local groupoid axioms
for a representative $G\rightrightarrows A$ of the symplectic microgroupoid. 
\end{proof}

The next proposition shows that $M$ is the {}``object'' component
of a functor 
\[
\mon:\ExtMicroSymplGpd\longrightarrow\Mon{\ExtMicroSympl}.
\]

\begin{prop}
\label{pro: lag. gpd = mon. morph.}Let $[G,A]\rightrightarrows A$
and $[H,B]\rightrightarrows B$ be two symplectic microgroupoids.
There is a one-to-one correspondence between the set of monoid morphisms
\begin{eqnarray*}
([L],\phi):\mon([G,A]) & \longrightarrow & \mon([H,B])
\end{eqnarray*}
and the set of lagrangian submicrogroupoids of the form
\begin{eqnarray*}
\Big(([L,\graph\phi]\rightrightarrows\graph\phi\Big) & \subset & \Big([\overline{G}\times H,A\times B]\rightrightarrows A\times B\Big),
\end{eqnarray*}
where $\phi:B\rightarrow A$ is a smooth map. 
\end{prop}

\begin{proof}
(1) Suppose that $[L,\graph\phi]\rightrightarrows\graph\phi$ is a
lagrangian submicrogroupoid as above. Let us show that the corresponding
symplectic micromorphism $([L],\phi)$ is a monoid morphism. The unitality
axiom \eqref{morph:neut-1} follows from the fact that there a unique
symplectic micromorphism from $E$ to any given symplectic microfold;
thus, $e_{B}=([L],\phi)\circ e_{A}$. As for the associativity axiom
\eqref{morph:ass-1}, we have to check that
\[
\big([L\circ\graph[m_{G}],\;\Delta_{A}\circ\phi\big)=\big([\graph[m_{H}]\circ(L\times L),\;(\phi\times\phi)\circ\Delta_{B}\big).
\]
Since the core maps coincide, it is enough to show that one of the
lagrangian submicrofolds is included in the other. We do this by taking
a representative $L\in[L]$ and by showing that $L(g_{1})L(g_{2})\subset L(g_{1}g_{2})$.
Take $h_{1}h_{2}\in L(g_{1})L(g_{2})$, which implies that $(g_{1},h_{1})\in L$
and $(g_{2},h_{2})\in L$. Since $L$ is a subgroupoid, we also have
$(g_{1}g_{2},h_{1}h_{2})\in L$, which, in other words, means that
$h_{1}h_{2}\in L(g_{1}g_{2})$. 

(2) Suppose now that $([L],\phi)$ is a monoid morphism. Let us prove
that $[L,\graph\phi]$ is a lagrangian submicrogroupoid over $\graph\phi$.
We now need to show that, for a representative $L\in[L],$ the images
$(s_{G}\times s_{H})(L)$ and $(t_{G}\times t_{H})(L)$ coincide with
$L_{0}=\graph\phi=L\cap(A\times B)$. For this, we repeat an argument
of Zakrzewski \cite[Lemma 2.4]{zakrzewski1990I}. First of all, we
immediately have 
\begin{gather*}
L_{0}=(s_{G}\times s_{H})(L_{0})\subseteq(s_{G}\times s_{H})(L),\\
L_{0}=(t_{G}\times t_{H})(L_{0})\subseteq(t_{G}\times t_{H})(L).
\end{gather*}
Let us check the first converse inclusion for the source maps. Take
$(g,h)\in L$; we want to show that $(s_{G}(g),s_{H}(h))\in L_{0}$.
Since $\{g\}=(\graph m_{A})(s_{G}(g),g)$, and $L$ is a monoid morphism,
it follows that
\begin{eqnarray*}
h\in L(g) & = & L((\graph m_{A})(s_{G}(g),g)),\\
 & = & (\graph m_{B})(L_{0}(s_{G}(g)),L(g)).
\end{eqnarray*}
This means that the source of $h$ is contained in the subset $L_{0}(s_{G}(g))$,
or, in other words, that $(s_{G}(g),s_{H}(h))\in L_{0}$. We argue
in a similar way for the target inclusion. 

It remains to show that $L\rightrightarrows\graph\phi$ is closed
under the groupoid product: Take $(g_{1},h_{1})\in L$ and $(g_{2},h_{2})\in L$,
then $h_{1}h_{2}\in L(g_{1})L(g_{2})=L(g_{1}g_{2})$ since $L$ is
a monoid map, which means that $(h_{1}h_{2},g_{1}g_{2})\in L$ and
which ends the proof. 
\end{proof}

From Proposition \ref{pro: lag. gpd = mon. morph.}, we see
that the functor $\mon$ is a full embedding of $\ExtMicroSymplGpd$
into the category of monoids in the microsymplectic category.

\subsection{Equivalence of categories}

Composing now the full embedding $M$ constructed in the previous
paragraph with the symplectization functor (which is an equivalence
of categories) 
\[
\poisson\overset{\Sigma}{\longrightarrow}\ExtMicroSymplGpd\overset{\mon}{\longrightarrow}\Mon{\ExtMicroSympl},
\]
we obtain a full embedding $Z=M\circ\Sigma$, which we call the \textbf{Zakrzewski
functor}. 

At this point, we can ask if $Z$ is an equivalence of categories,
or, in other words, if every monoid in the microsymplectic category
arises from a symplectic microgroupoid. As the following example makes
it clear, this is not true in the macro case. Namely, a symplectic
groupoid has an additional structure that is not encoded in its corresponding
monoid: The graph of its inverse map endows the corresponding monoid
with the structure of a $\star$-monoid (see \cite{weinstein1991,zakrzewski1990II}
for details). The monoid, in the macro case, merely encodes the information
relative to the product, the source and the target maps of the symplectic
groupoid as well as its unit space. 

\begin{example}
\label{exam: sing. gpd}We give here a monoid object in $\ExtSympl$
that does not come from a symplectic groupoid. Consider the operation
$m(x_{1},x_{2})=x_{1}x_{2}$ on real numbers. We regard its cotangent
lift $\Cot m$ as a canonical relation from $\R^{2}\otimes\R^{2}$
to $\R^{2}$, where the symplectic structure on $\R^{2}$ comes from
the identification $\Cot\R_{x}\simeq\R_{x}\oplus\R_{y}$. Explicitly,
we have
\[
\Cot m:=\Big\{\Big((x_{1},yx_{2}),\,(x_{2},yx_{1}),\,(x_{1}x_{2},y)\Big):\: x_{1},x_{2},y\in\R\Big\}.
\]
One may verify directly the associativity of $\Cot m$, which follows
from that of $m$. There is a unit, which is given by the lagrangian
submanifold $E=\{(1,y):\; y\in\R\}$. We still have a source and a
target, which coincide here
\[
s=t=\pi:\R^{2}\rightarrow E;\;(x,y)\mapsto(1,xy).
\]
The restriction of $\Cot m$ to the fibers of $\pi$ over points $(1,\alpha)$
with $\alpha\neq0$ can be seen as the graph of a product
\[
\Big(x_{1},\frac{\alpha}{x_{1}}\Big)\bullet\Big(x_{2},\frac{\alpha}{x_{2}}\Big)=\Big(x_{1}x_{2},\frac{\alpha}{x_{1}x_{2}}\Big).
\]
Actually, this product endows the fiber $\pi^{-1}(1,\alpha)$, which
is the hyperbola $H_{\alpha}=\{(x,y):\: xy=\alpha\}$, as represented
in Figure \ref{fig:Hyperbole}, with a group structure. The unit is
the element $(1,\alpha)$, and the inverse is
\[
\Big(x,\frac{\alpha}{x},\Big)^{-1}=\Big(\frac{1}{x},\alpha x\Big).
\]
Although we are very close to having a symplectic groupoid, there
is a ``singular'' fiber; $\pi^{-1}((1,0))$ is the union of the two
intersecting lines $l_{x}=\{(x,0):\: x\in\R\}$ and $l_{y}=\{(0,y):\: y\in\R\}$.
The product is defined on $l_{x}\backslash\{(0,0)\}$ as the usual
real multiplication, but it is not defined at all on $l_{y}\backslash\{(0,0)\}$.
Moreover, the product is multi-valued at $(0,0)$; namely,
\[
(0,0)\bullet(0,0)=\{(0,y):\: y\in\R\}.
\]
Observe that the problematic points (i.e. the whole line $l_{y}$)
coincide with the locus of points where the inverse is not defined.

\begin{figure}
\centering{}\includegraphics[scale=0.5]{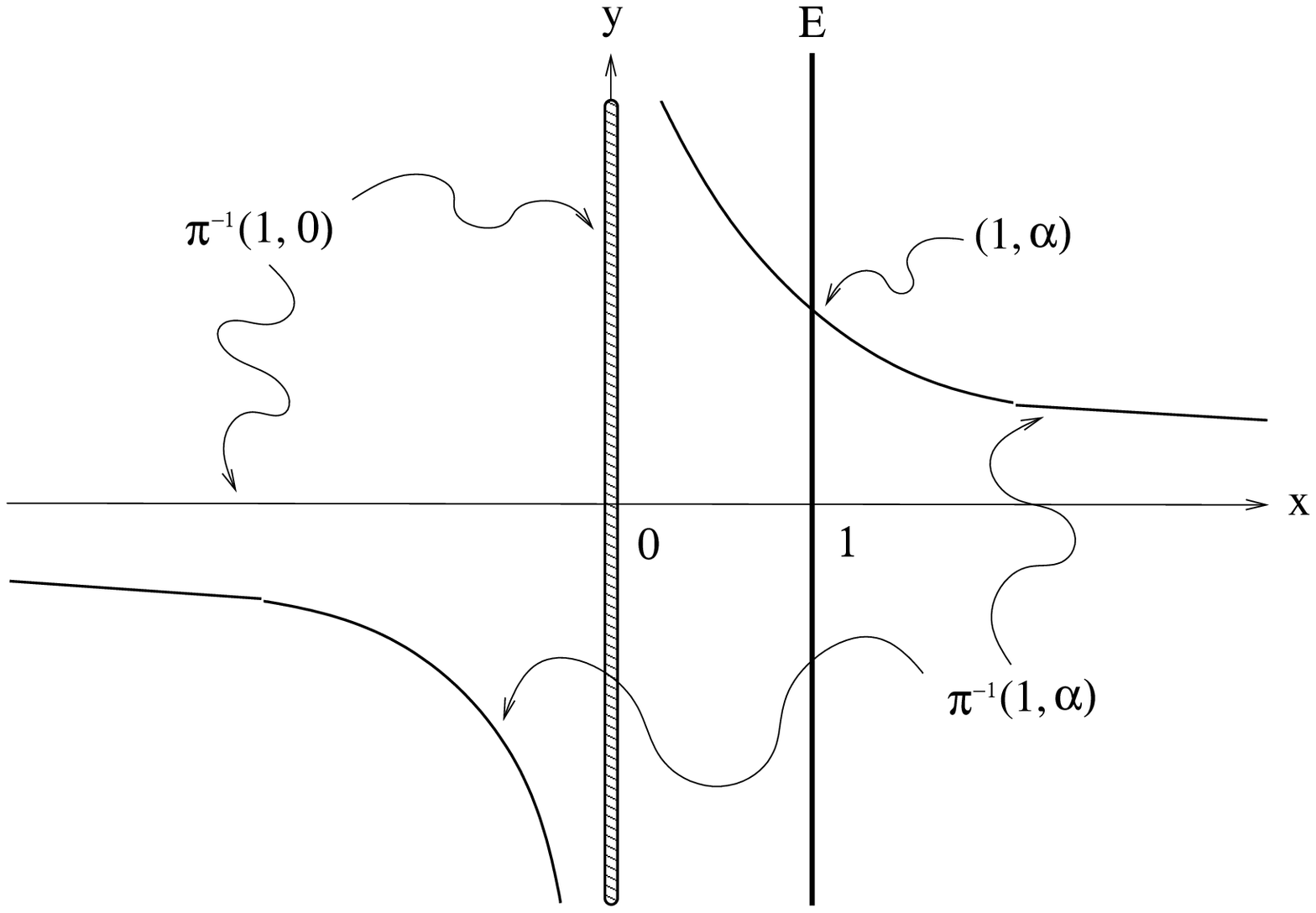}\caption{\label{fig:Hyperbole}}
\end{figure}

\end{example}

The situation in the micro-world is different. The following example
illustrates the fact that, even though a monoid object in the extended
symplectic category may not arise from a symplectic groupoid, its
restriction to a (micro) neighborhood of its unit space yields a monoid
object in the microsymplectic category, which itself comes from a
symplectic microgroupoid. This is essentially due to the fact that
all the structure maps (including the inverse map) of a local symplectic
groupoid can be recovered from the source map alone (\cite{CDW1987,karasev1987}). 

\begin{example}
Let us look at Example \ref{exam: sing. gpd} again. The monoid $(\R^{2},\Cot m,E)$
in $\ExtSympl$ fails to come from a symplectic groupoid because the
natural inverse map
\[
i(x,y)=(\frac{1}{x},yx^{2})
\]
can not be extended to the line $l_{y}$ of the ``singular'' fiber
$\pi^{-1}(1,0)=l_{x}\cup l_{y}$. This prevents this monoid from having
a $*$-monoid structure (see \cite{Canez2011,zakrzewski1990I,zakrzewski1990II}
for a definition of a $*$-monoid). However, if we look at the restriction
of $\Cot m$ to a neighborhood
\[
U_{\epsilon}:=\Big\{(1+\eta,y):\:\eta\in(-\epsilon,\epsilon),\, y\in\R\Big\},\quad0<\epsilon<1,
\]
we obtain a local symplectic groupoid over $E$ and a corresponding
symplectic micromorphism
\[
([\Cot m],\Delta_{E}):[\R^{2},E]\otimes[\R^{2},E]\rightarrow[\R^{2},E],
\]
turning $[\R^{2},E]$ into a monoid object in the microsymplectic
category. 

First of all, we will use the coordinates $q:=y$ on $E$ and $p:=x-1$
on the vertical fiber so that a point $(x,\frac{\alpha}{x})$ in the
$xy-$coordinates reads $(x-1,\frac{\alpha}{x})$ in the $pq$-coordinates.
Now the restriction of $\Cot m$ to $U_{\epsilon}$ becomes
\[
\Big\{\big((p_{1},q(p_{2}+1)\big),\big(p_{2},q(p_{1}+1)\big),\big(p_{1}+p_{2}+p_{1}p_{2},q\big):\:0<p_{1,},p_{2}<1,\, q\in\R\Big\}.
\]
In these coordinates, the source and target, which coincide, of the
local symplectic groupoid $U_{\epsilon}\rightrightarrows E$ are given
by the map $\pi(p,q)=q(p+1)$. Now, for each $\alpha\in\R$, the fiber
\[
\pi^{-1}(\alpha)=\left\{ \left(p,\frac{\alpha}{p+1}\right):\: p\in(-\epsilon,\epsilon)\right\} 
\]
is an honest submanifold, and the groupoid product and the inverse
map, restricted to this fiber, are given by
\begin{eqnarray*}
\left(p_{1},\frac{\alpha}{p_{1}+1}\right)\bullet\left(p_{2},\frac{\alpha}{p_{2}+1}\right) & = & \left(p_{1}+p_{2}+p_{1}p_{2},\;\frac{\alpha}{p_{1}+p_{2}+p_{1}p_{2}+1}\right),\\
\left(p,\frac{\alpha}{p+1}\right)^{-1} & = & \left(\frac{-p}{p+1},\frac{\alpha}{\left(\frac{-p}{p+1}\right)+1}\right).
\end{eqnarray*}
Figure \ref{fig:loc hyperboles} shows the local groupoid fibration
generated by the family of hyperbolae $H_{\alpha}$, for $\alpha\neq0$,
together with the line $l_{x}$ in a neighborhood of the unit space
$E$ in Example \ref{exam: sing. gpd}.

\begin{figure}
\begin{centering}
\includegraphics[scale=0.5]{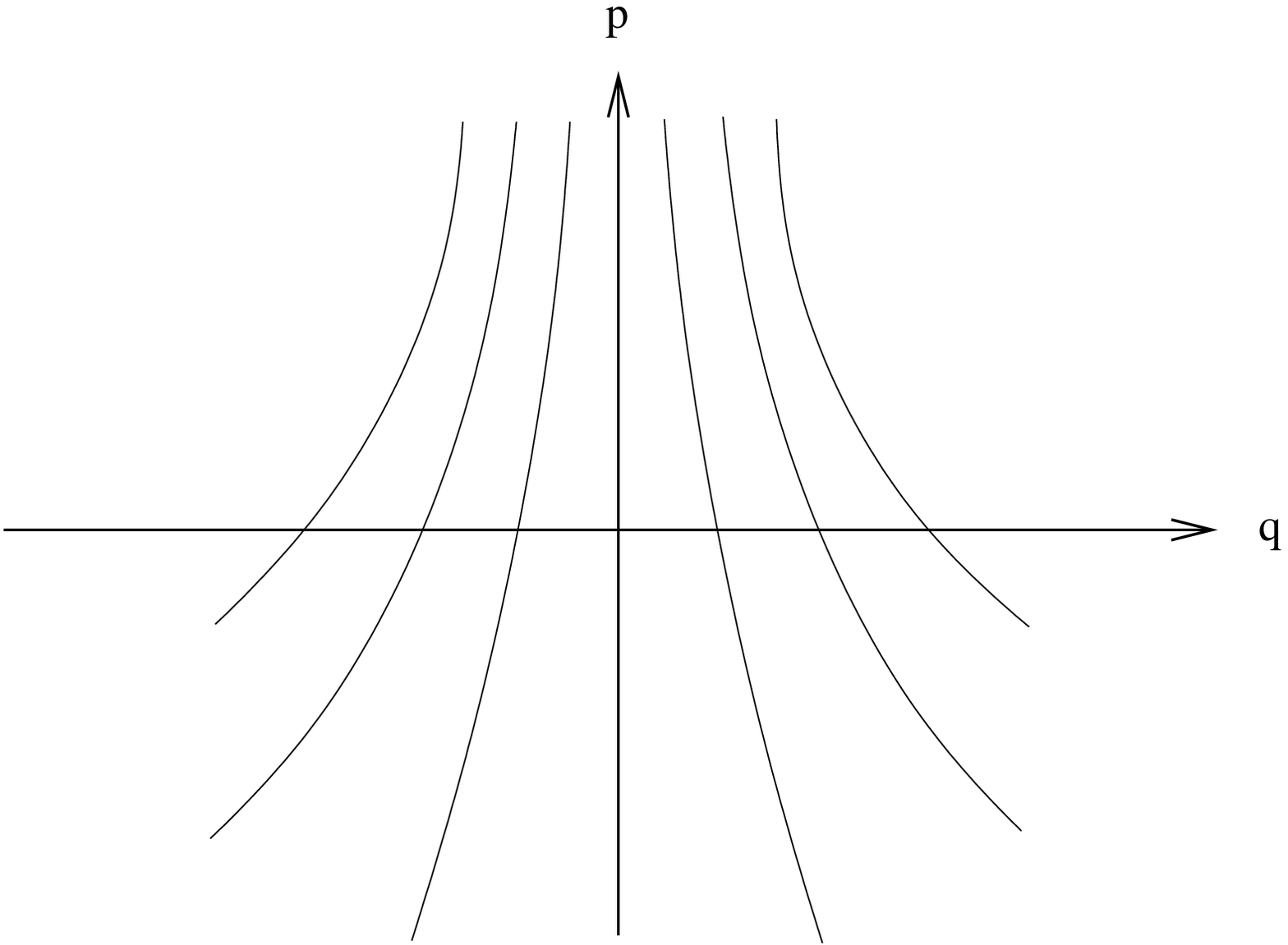}\caption{\label{fig:loc hyperboles}}
\par\end{centering}
\end{figure}

\end{example}

Actually, in the micro case, we have the following proposition:

\begin{prop} \label{prop: monoids=groupoids}
Every monoid in the microsymplectic category comes from a symplectic microgroupoid. 
\end{prop}

\begin{proof}
To begin with, we show that $\mu=\graph[m]$, where $[m]$ is the
germ of a map $m:\mathcal{C}\rightarrow G$, where $\mathcal{C}$
is a representative of a coisotropic submicrofold $[\mathcal{C},\graph\id_{A}]\subset[G\times G,A\times A]$.
To see this, we will use the result \cite[Theorem 27]{CDW2011}, which,
applied to the symplectic micromorphism
\[
\mu=([V],\phi):[G^{2},A^{2}]\rightarrow[G,A],
\]
states the following: For any collection $\mathcal{F}A:=\{[\mathcal{F}_{x}A,\{x\}]\}_{x\in A}$
of lagrangian submanifold germs in $[G,A]$ transverse to $A$, there
is a corresponding collection
\[
\mathcal{F}A^{2}:=\big\{\big[\mathcal{F}_{x}A^{2},\{\phi(x)\}\big]\big\}_{x\in A}
\]
of lagrangian submanifold germs in $[G^{2},A^{2}]$, each of which
is transverse to the core $A^{2}$, together with a collection of
map germs
\[
m_{x}:\mathcal{F}_{x}A^{2}\rightarrow\mathcal{F}_{x}A,\quad x\in A,
\]
such that, for suitable representatives,
\begin{eqnarray*}
\graph m_{x} & = & V\cap(\mathcal{F}_{x}A^{2}\times\mathcal{F}_{x}A),\\
V & = & \bigcup_{x\in A}\graph m_{x}.
\end{eqnarray*}
The key point now is to note that the unitality axiom \eqref{ass:neut-1}
forces the core map of $\mu$ to be the diagonal map $\Delta_{A}:A\rightarrow A\times A$,
which is an embedding. This implies that the lagrangian germs in the
collection $\mathcal{F}A^{2}$ are all disjoint. Therefore, their
union
\[
\mathcal{C}:=\bigcup_{x\in A}\mathcal{F}_{x}A^{2}
\]
is a germ of a coisotropic submanifold around $\{(x,x):x\in A\}$,
on which the collection $\{m_{x}\}_{x\in A}$ defines a map $m:\mathcal{C}\rightarrow A$
such that $V=\graph m$. This is the (local) symplectic groupoid product,
and $\mathcal{C}$ is the space of composable pairs. 

Now, let us extract the source, target and inverse maps while verifying
the local groupoid axioms. The associativity axiom \eqref{ass:ass-1}
in term of $m$ yields directly that 
\begin{equation}
(g_{1},g_{2})\in\mathcal{C},\,(m(g_{1},g_{2}),g_{3})\in\mathcal{C}\quad\Rightarrow\quad(g_{2},g_{3})\in\mathcal{C},\,(g_{1},m(g_{2},g_{3}))\in\mathcal{C}.\label{N1}
\end{equation}
Moreover in this case we have that
\begin{equation}
m(m(g_{1},g_{2}),g_{3})=m(g_{1},m(g_{2},g_{3})).\label{N2}
\end{equation}
Since $\graph[m]$ is a symplectic micromorphism with core map $\Delta:A\rightarrow A\times A$,
we have that $\graph m\cap(A\times A\times G)=\graph\Delta$. Hence,
for $x_{1},x_{2}\in A$, we have that $(x_{1},x_{2})\in\mathcal{C}$
iff $x_{1}=x_{2}=x$ and then $m(x,x)=x$. 

We can extract the source and target maps from the fact that the composition
of symplectic micromorphisms is always monic (see \cite{CDW2011})
when applied to the compositions involved in the unitality axiom \eqref{ass:neut-1}.
Namely, consider the reduction
\[
\pi:G\times\Delta_{G\times G}\times G\rightarrow G\times G
\]
associated with the composition
\[
G\simeq\{*\}\times G\overset{e_{A}\times\id}{\longrightarrow}G\times G\overset{\graph m}{\longrightarrow}G.
\]
The monicity of the composition can be expressed by saying that the
restriction $\pi'$ of $\pi$ to 
\[
K:=(A\times\Delta_{G}\times\graph m)\cap(G\times\Delta_{G\times G}\times G)
\]
is a diffeomorphism on its image $\graph m\circ(e_{A}\otimes\id)$.
Now, using the identification $K\simeq(A\times G)\cap\mathcal{C}$,
the restriction $\pi'$ reads
\[
\pi':(A\times G)\cap\mathcal{C}\rightarrow G\times G:\:(a,g)\rightarrow(g,m(a,g))=(g,g).
\]
Taking the inverse of $\pi'$ on its image (which is the diagonal
in $G\times G$ by the unitality axiom), we obtain a smooth surjective
map $s:G\rightarrow A$ for which $s(g)$ is the unique element in
$A$ such that
\begin{equation}
(s(g),g)\in\mathcal{C}\quad\textrm{and}\quad m(s(g),g)=g\quad\textrm{for all}\quad g\in G.\label{N4}
\end{equation}
In other words, we have obtained the source map of the local symplectic
groupoid. The same line of argument for the other composition $\graph m\circ(\id\otimes e_{A})$
delivers the target map $t:G\rightarrow A$. Let us check now that
these maps behave the way they should with respect to the product
$m$; for this, we use here essentially the same arguments as in Zakrzewski
\cite[Lemma 2.3]{zakrzewski1990I}. Let $(g_{1},g_{2})\in\mathcal{C}$,
then
\[
\begin{gathered}m(g_{1},g_{2})=m(m(s(g_{1}),g_{1}),g_{2})=m(s(g_{1}),m(g_{1},g_{2})),\\
m(g_{1},g_{2})=m(g_{1},m(g_{2},t(g_{2})))=m(m(g_{1},g_{2}),t(g_{2})),
\end{gathered}
\]
which implies, by uniqueness of the source and target map, that
\begin{equation}
s(m(g_{1},g_{2}))=s(g_{1})\quad\textrm{and}\quad t(m(g_{1},g_{2}))=t(g_{2}).\label{N5}
\end{equation}
Since we also have, for $(g_{1},g_{2})\in\mathcal{C}$, that 
\[
m(g_{1},g_{2})=m(m(g_{1},t(g_{1})),g_{2})=m(g_{1},m(t(g_{1}),g_{2})),
\]
which implies that $(t(g_{1}),g_{2})\in\mathcal{C}$, we obtain that
\begin{equation}
\forall(g_{1},g_{2})\in\mathcal{C},\quad t(g_{1})=s(g_{2})\label{N6}
\end{equation}
because $s(g_{2})$ is the unique element $x\in A$ such that $(x,g_{2})$
is composable. 

To extract the inverse map, we consider the restriction $m_{x}$ of
the product $m$ to the fiber $\mathcal{F}_{x}A^{2}$ and the implicit
equation $m_{x}(g_{1},g_{2})=x$. Since $m_{x}(x,x)=x$ and $\partial_{g_{1}}m_{x}(x,x)=\partial_{g_{2}}m_{x}(x,x)=\id$,
we can apply the implicit function theorem in both arguments, obtaining
two families of smooth functions, $\{i_{x}\}_{x\in A}$ and $\{j_{x}\}_{x\in A}$,
such that
\[
m_{t(g)}(i_{t(g)}(g),g)=t(g)\textrm{ and }m_{s(g)}(g,j_{s(g)}(g))=s(g).
\]
The collection $\{\mathcal{F}_{x}A^{2}\}_{x\in A}$ of lagrangian
submanifolds is a partition of $\mathcal{C}$, and so we can glue
the two collections $\{i_{x}\}_{x\in A}$ and $\{j_{x}\}_{x\in A}$
into smooth functions $i,j:G\rightarrow G$ such that
\begin{equation}
m(i(g),g)=t(g)\textrm{ and }m(g,j(g))=s(g).\label{N7}
\end{equation}
Since, by construction, $(i(g),g),\,(g,j(g))\in\mathcal{C}$, we also
have that
\begin{equation}
t(i(g))=s(g)\textrm{ and }t(g)=s(j(g)).\label{N8}
\end{equation}
 Let us show that $i$ and $j$ coincide. Namely, $(m(i(g),g),j(g))\in\mathcal{C}$
because $m(i(g),g)=t(g)=s(j(g))$ and thus
\[
j(g)=m(m(i(g),g),j(g))=m(i(g),m(g,j(g))=i(g).
\]
\end{proof}

As explained in Section \ref{sub:Categories-of-monoids}, the category
of monoids in the microsymplectic category inherits the structure
of a symmetric monoidal category. 

\begin{thm}
The functor $\mon:\ExtMicroSymplGpd\rightarrow\Mon{\ExtMicroSympl}$ is an equivalence of symmetric monoidal categories.
\end{thm}

\begin{proof}
Proposition \ref{prop: monoids=groupoids} tells us that $M$
is an isomorphism on the objects and Proposition \ref{pro: lag. gpd = mon. morph.}
tells us that $M$ is full and faithful. Hence, we have an isomorphism
of categories. 

Given two symplectic microgroupoids $[G,A]\rightrightarrows A$ and
$[H,B]\rightrightarrows B$ the groupoid operation in their tensor
product is given by
\[
[m_{G\times H}]:=([m_{G}]\times[m_{H}])\circ(\id_{G}\times\epsilon_{G,H}\times\id_{H}).
\]
Therefore, the image by $M$ of the symplectic microgroupoid product
is the monoid
\[
\Big([G\times H,A\times B],\,\mu,\, e_{A\times B}\Big),
\]
where
\[
\mu:=(\graph[m_{G}]\otimes\graph[m_{H}])\circ(\id_{[G,H]}\otimes\graph\epsilon_{G,H}\otimes\id_{H}).
\]
From Section \ref{sub:Categories-of-monoids}, we see that this monoid
coincides with the tensor product of the monoids; i.e., 
\[
M([G\times H,A\times B)=M([G,A])\otimes M([H,B]).
\]
 A similar argument shows that we also have this property for the
morphisms; i.e.,
\[
M([V\otimes W],\phi\times\psi)=M([V],\phi)\otimes M([W],\psi).
\]
Moreover, the unit symplectic microgroupoid $[\{0\}\times\{\star\},\{\star\}]\rightrightarrows\{\star\}$
is sent on the unit monoid by $M$. Thus, we have an equivalence of
monoidal categories. Checking that $M$ also respects the symmetries
is straightforward. 
\end{proof}

Now, composing the following equivalences
\[
\poisson\overset{\Sigma}{\longrightarrow}\ExtMicroSymplGpd\overset{\mon}{\longrightarrow}\Mon{\ExtMicroSympl},
\]
we obtain the main theorem of this paper:

\begin{thm} \label{thm: main}
The symmetric monoidal category of Poisson manifolds
and Poisson maps is equivalent to the symmetric monoidal category
of monoids in the microsymplectic category.
\end{thm}

\end{document}